\documentclass[12pt,a4paper]{amsart}
\usepackage{a4wide}
\usepackage{amsmath,amsfonts,amssymb,amscd,amsthm,latexsym,mathrsfs,comment}
\usepackage{xcolor}

\theoremstyle{plain}
\newtheorem{Thm}{Theorem}[section]
\newtheorem{Lem}[Thm]{Lemma}
\newtheorem{Prop}[Thm]{Proposition}
\newtheorem{Cor}[Thm]{Corollary}
\theoremstyle{definition}

\newtheorem{Rem}[Thm]{Remark}

\numberwithin{equation}{section}

\DeclareMathOperator{\End}{End}
\DeclareMathOperator{\Ker}{Ker}
\DeclareMathOperator{\tr}{tr}
\DeclareMathOperator{\dom}{\mathbf{d}}
\DeclareMathOperator{\ran}{\mathbf{r}}
\DeclareMathOperator{\Hdom}{dom}

\newcommand{\bbK}{\mathbb{K}}
\newcommand{\bbS}{\mathbb{S}}

\newcommand{\olbbS}{\overline{\mathbb{S}}}
\newcommand{\olTheta}{\overline{\Theta}}
\newcommand{\olK}{\overline{K}}
\newcommand{\olS}{\overline{S}}
\newcommand{\olepsilon}{\overline{\epsilon}}
\newcommand{\olphi}{\overline{\phi}}
\newcommand{\oltheta}{\overline{\theta}}
\newcommand{\olxi}{\overline{\xi}}
\newcommand{\olf}{\overline{f}}
\newcommand{\olh}{\overline{h}}
\newcommand{\Rho}{\mathrm{P}}

\newcommand{\lsd}[2]{#1\kern -1pt\mathrel{*^{\lambda}}\kern -1pt#2}
\newcommand{\rsd}[2]{#1\kern -1pt\mathrel{\bowtie}\kern -1pt#2}
\newcommand{\gsd}[2]{#1\kern -1pt\mathrel{\star}\kern -1pt#2}
\newcommand{\Hwr}[2]{#1\kern 0pt\mathrel{\operatorname{Wr}^H}\kern -1pt#2}
\newcommand{\lwr}[2]{#1\kern 0pt\mathrel{\operatorname{Wr}^{\lambda}}\kern -1pt#2}
\newcommand{\Hwrp}[2]{#1\kern 0pt\mathrel{\operatorname{Wr}^H_\eta}\kern -1pt#2}
\newcommand{\rwr}[2]{#1\kern 0pt\mathrel{\operatorname{Wr}^{\bowtie}}\kern-1pt#2}

\newcommand{\act}[2]{{^{#1}\kern -2pt {#2}}}
\newcommand{\acth}[2]{{^{#1}\kern -1pt {#2}}}

\newcommand{\rtheta}{\mathrel{\theta}}

\begin{document}
	
\title[Normal extensions and full restricted semidirect products]
{Normal extensions and full restricted semidirect products of inverse semigroups}


\author{M\'aria B.\ Szendrei}
\address{Bolyai Institute, University of Szeged, Aradi v\'ertan\'uk tere 1, Szeged, Hungary, H-6720}
\email{m.szendrei@math.u-szeged.hu}


\thanks{{\sc Dedicated to E.\;S.\;Ljapin on the 110th anniversary of his birth.\\}
	Research is partially supported by the National Research, Development and Innovation Office, grants NKFI/K128042, NKFI/K138892 and NKFI/K142484.
	}

\begin{abstract}
	We characterize the normal extensions of inverse semigroups isomorphic to full restricted semidirect products, and present a Kalou\v{z}nin--Krasner theorem which holds for a wider class of normal extensions of inverse semigroups than that in the well-known embedding theorem due to Billhardt, and also strengthens that result in two respects.
	First, the wreath product construction applied in our result, and stemmming from Houghton's wreath product, is a full restricted semidirect product not merely a $\lambda$-semidirect product.
	Second, the Kernel classes of our wreath product construction are direct products of some Kernel classes of the normal extension to be embedded rather than only inverse subsemigroups of the direct power of its whole Kernel.
\end{abstract}

\maketitle

\section{Introduction}

The construction of forming a semidirect product of groups can be naturally generalized for semigroups if one allows actions by endomorphisms instead of actions by automorphisms. 
This construction and a similar generalization of wreath product of groups play fundamental roles in the theory of semigroups, and especially, of finite semigroups.
A well-known result of the theory of inverse semigroups which is due to O'Carroll \cite{OCarroll} establishes that each $E$-unitary inverse semigroup, that is, each normal extension of a semilattice by a group, is embeddable in a semidirect product of a semilattice by a group.
This result is generalized by Billhardt \cite{BEom} for extensions of Clifford semigroups by groups.
However, semidirect and wreath products of inverse semigroups fail to be inverse in general except when the second factor is a group.
To overcome this difficulty, 
Billhardt \cite{Blambda} (see also \cite[Section 5]{LawInvSg}) introduced modified versions of these constructions appropriate for inverse semigroups, and he 
called them $\lambda$-semidirect and $\lambda$-wreath products.
The action in a $\lambda$-semidirect product is by endomorphisms in the same way as in a usual semidirect product of semigroups but both the underlying set and the multiplication rule are modified.
In the same paper, Billhardt introduced a class of congruences 
and proved that a normal extension of inverse semigroups determined by such a congruence is embeddable in a $\lambda$-wreath product.
Lawson named such congruences in his monograph \cite{LawInvSg} Billhardt congruences. 
For example, idempotent separating congruences are Billhardt, see \cite{BKK} (or \cite[Chapter 5]{LawInvSg}). 
In the same paper, an analogue of a split group extension is also introduced as a normal extension determined by a split Billhardt congruence, and these extensions are proved to be isomorphic to full restricted semidirect products. 
A full restricted semidirect product is an inverse subsemigroup of a $\lambda$-semidirect product provided that the action fulfils additional conditions. 
It is important to notice that a full restricted semidirect product is a closer analogue of a semidirect product of groups than a $\lambda$-semidirect product
since a full restricted semidirect product is a normal extension of its first factor by the second but this is not the case with a $\lambda$-semidirect product in general, except when the second factor is a group. 

Much earlier than $\lambda$-semidirect product was introduced, an embedding of an idempotent separating extension into another kind of wreath product
was presented by Houghton \cite{Hough}, see also \cite[Section 11.2]{Meldrum}. 
The notion of Houghton's wreath product also stems from the notion of the standard wreath product of groups but in a way that the direct power of the first factor to the second is replaced by a semilattice of the direct powers of the first factor to the principal left ideals of the second.
Houghton's wreath product and $\lambda$-semidirect product are closely related to each other, see \cite{SzMregsem} and \cite[Section 5.5]{LawInvSg}.  

The embedding results mentioned so far mimic the group case also in the sense that only the Kernel and the factor of an extension are taken into consideration.
However, among inverse semigroups, it is more natural to consider the Kernel and the trace of a congruence simultaneously rather than only the Kernel, since distinct congruences might have the same Kernels.
In this paper we are interested in embeddability of normal extensions in a $\lambda$-semidirect product in such a way that trace is also `preserved'. 
Motivated by Billhardt's statement \cite[Lemma 3]{BKK} which implies that each $\lambda$-semidirect product $\lsd{K}{T}$ of $K$ by $T$ is naturally embeddable in a full restricted semidirect product of the Kernel of $\lsd{K}{T}$ by $T$, we focus on embeddability of normal extensions in a full restricted semidirect product.
It will turn out that Houghton's wreath product can be naturally adapted to our purposes.  

The main results of the paper are in Sections \ref{sect:char-rsd} and \ref{sect:emb-rsd}.
In Section \ref{sect:prelim}, we mention the main facts on inverse semigroups which are needed in the paper.
Moreover, we slightly extend \cite[Lemma 3]{BKK} mentioned in the previous paragraph and make it more explicit by noticing that, for any inverse semigroups $K$ and $T$, a $\lambda$-semidirect product $\lsd{K}{T}$ and the full restricted semidirect product of the Kernel of $\lsd{K}{T}$ by $T$ constructed from it in \cite[Lemma 3]{BKK} are, actually, isomorphic to each other.
We start Section \ref{sect:char-rsd} by giving an alternative system of axioms for an action needed in the definition of a full restricted semidirect product which allows us to simplify later arguments.
The goal of the section is to characterize the normal extensions isomorphic to full restricted semidirect products.
We introduce classes of congruences called almost Billhardt and split almost Billhardt congruences which generalize Billhardt and split Billhardt congruences, respectively, and we prove that a normal extension is isomorphic to a full restricted semidirect product precisely when it is defined by a split almost Billhardt congruence.
In Section \ref{sect:emb-rsd} we present a Kalou\v{z}nin--Krasner theorem for each normal extension defined by an almost Billhardt congruence $\theta$ which embeds such a normal extension into a full restricted semidirect product whose Kernel classes are direct products of idempotent $\theta$-classes. 
The `general view', formulated also in \cite[p.\ 156]{LawInvSg}, that Billhardt congruences are intimately connected with $\lambda$-semidirect products and split Billhard congruences with full restricted semidirect products, is put by this theorem in a different light.
The full restricted semidirect product appearing in our result is an inverse subsemigroup of Houghton's wreath product of the Kernel of the normal extension by its factor which corresponds to the respective normal extension triple not merely to the Kernel and the factor.

\section
{Preliminaries and an initial observation}
\label{sect:prelim}

In this section we outline the most important facts needed in the paper on normal extensions of inverse semigroups in general and on three constructions, namely $\lambda$-semidirect product, full restricted semidirect product and Houghton's wreath product.
A short introduction to translations is also included.
For more details, the reader is referred to the monographs by Lawson \cite[Sections 5.1 and 5.3]{LawInvSg}, Meldrum \cite[Section 11.2]{Meldrum} and Petrich \cite[Section VI.6]{PetInvSg}.
Additionally, we slightly strengthen a statement due to Billhardt \cite{BKK} which presents a close connection between embeddability of a normal extension in a $\lambda$-semidirect product and in a full restricted semidirect product.

Our notation mainly follows that in \cite{LawInvSg}.
In particular, functions are written as left operators, and are composed from the right to the left.
The only exceptions are right translations which are written as right operators, and their composition is carried over from the left to the right.
It is also worth calling the attention in advance, that the terms `kernel' and `Kernel' will be used in the following manner, see \cite{LawInvSg}.
The \emph{Kernel of a congruence} $\rho$ on an inverse semigroup $S$, denoted by  $\Ker \rho$, is the inverse subsemigroup of $S$ consisting of the elements $\rho$-related to an idempotent.
The \emph{kernel of a homomorphism} $\phi\colon S\to T$ between inverse semigroups $S,T$, denoted by $\ker \phi$, is the congruence on $S$ induced by $\phi$, and    
the \emph{Kernel of} $\phi$, denoted by $\Ker \phi$, is the Kernel of the congruence $\ker \phi$. 

Most of the facts mentioned in this section are applied in the rest of the paper without reference.

\subsection*{Normal extension}

Let $K$ be an inverse semigroup, let $E$ be a semilattice, and consider a surjective homomorphism $\eta\colon K\to E$.
Then the \emph{semilattice decomposition corresponding to} $\eta$ is 
$K = \bigcup_{e\in E} K_e$
where 
\begin{equation*}
	K_e = \{a\in K : \eta(a) = e\}\quad (e\in E)
\end{equation*}
are the $(\ker \eta)$-classes which are inverse subsemigroups in $K$, and $K_eK_f\subseteq K_{ef}$ for any $e,f\in E$.
Conversely, such a decomposition determines a surjective homomorphism 
\begin{equation*}
	\eta\colon K\to E\ \ \hbox{where}\ \ \eta(a) = e\ \hbox{if}\ a\in K_e,
\end{equation*}
so that these two formulations are equivalent.
We use these alternatives simultaneously.

Now let $K$ and $T$ be inverse semigroups and $\eta\colon K\to E(T)$ a surjective homomorphism.
Then $(K,\eta,T)$ is called a \emph{normal extension triple}, and an inverse semigroup $S$ is said to be a \emph{normal extension of $K$ by $T$ along $\eta$} if 
there exists an embedding (i.e., an injective homomorphism) $\iota\colon K\to S$  and a surjective homomorphism $\tau\colon S\to T$ such that $\iota(K)=\Ker \tau$ and $\tau\iota = \eta$.
Such a triple $(\iota,S,\tau)$ is called a \emph{solution of the normal extension problem for the triple $(K,\eta,T)$}. 
Two solutions $(\iota,S,\tau)$ and $(\iota',S',\tau')$ for $(K,\eta,T)$ are said to be \emph{equivalent} if there is an isomorphism $\phi\colon S\to S'$ such that 
$\phi\iota = \iota'$ and $\tau'\phi = \tau$.

Somewhat more generally, now let  
$(K,\eta,T)$ and $(K',\eta',T')$ be normal extension triples, and consider a solution  
$(\iota,S,\tau)$ and $(\iota',S',\tau')$, respectively, for each of them. 
If there exists a triple $(\chi,\phi,\psi)$ of embeddings (resp.\ isomorphisms) $\chi\colon K\to K'$, $\phi\colon S\to S'$, $\psi\colon T\to T'$ such that
$\iota' \chi = \phi \iota$ and $\tau' \phi = \psi \tau$
then we say that the triple $(\chi,\phi,\psi)$ is an 
\emph{embedding from $(\iota,S,\tau)$ into  $(\iota',S',\tau')$} (resp.\ \emph{isomorphism from $(\iota,S,\tau)$ onto $(\iota',S',\tau')$}).
It is routine to see that if $(\chi,\phi,\psi)$ is such an embedding (resp.\ 
isomorphism) then $\chi$ and $\psi$ are uniquely determined by $\phi$, and the relations
\begin{equation*}
\phi\iota(K)\subseteq \iota'(K')\ (\hbox{resp.}\ 
\phi\iota(K) = \iota'(K'))\quad \hbox{and}\quad 
\ker \tau = \ker \tau'\phi
\end{equation*}
hold. 
Conversely, if $\phi\colon S\to S'$ is an embedding (resp.\ isomorphism) fulfulling 
these conditions 
then there exist appropriate $\chi$ and $\psi$ to form with $\phi$ an embedding from $(\iota,S,\tau)$ into  $(\iota',S',\tau')$ (resp.
isomorphism from $(\iota,S,\tau)$ onto $(\iota',S',\tau')$).
Based on this fact, we will consider an embedding (resp.\ isomorphism) from $(\iota,S,\tau)$ into $(\iota',S',\tau')$ to be such a $\phi$ rather than the respective triple $(\chi,\phi,\psi)$.

Notice that if $(\iota,S,\tau)$ is a solution of the normal extension problem for the normal extension triple $(K,\eta,T)$ then it is equivalent to the solution $(1_{\Ker \tau, S},S,(\ker \tau)^\natural)$ where $1_{\Ker \tau, S}$ stands for the function ${\Ker \tau\to S},\ a\mapsto a$.
Clearly, $T$ is isomorphic to $S/\ker \tau$, $K$ is isomorphic to $\Ker \tau$ and $\tr \eta = \tr \tau$.
Therefore, up to isomorphism, we can restrict our attention to the solutions for $(K,\eta,T)$ which are of the form $(1_{\Ker \theta, S},S,\theta^\natural)$ where $\theta$ is a congruence on $S$ such that  $\tr \theta = \tr \eta$.
Since a congruence on $S$ is uniquely determined by its Kernel and trace such a solution will be simply denoted by $(S,\theta)$.

If $(S,\theta)$ and $(S',\theta')$ are solution for $(K,\eta,T)$ and $(K',\eta',T')$, respectively, then an embedding (resp.\ isomorphism) $\phi\colon S\to S'$ is an embedding (resp.\ isomorphism) from $(S,\theta)$ into (resp.\ onto) $(S',\theta')$ if and only if
\begin{eqnarray*}
&\ &\phi(K)\subseteq K'\ (\hbox{resp.}\ \phi(K) = K'), \quad \hbox{and}\\
&\ &s\rtheta s'\quad \hbox{if and only if}\quad \phi(s) \rtheta' \phi(s')\quad \hbox{for every}\quad s,s'\in S.
\end{eqnarray*}

\subsection*{Constructions}

Let $K$ and $T$ be inverse semigroups. 
We say that \emph{$T$ acts on $K$ by endomorphisms} if a function $T\times K\to K,\ (t,a)\mapsto t\cdot a$ is given such that the transformations $\alpha_t\ (t\in T)$ of $K$ defined by $a\mapsto t\cdot a$ are endomorphisms and the function $T\to \End K,\ t\mapsto \alpha_t$ is a homomorphism.

If $T$ acts on $K$ by endomorphisms then the \emph{$\lambda$-semidirect product of $K$ by $T$ with respect to this action} is the inverse semigroup defined on the set
\begin{equation*}
	\lsd KT = \{(a,t)\in K\times T : a = \ran(t)\cdot a\}
\end{equation*}
by the operation
\begin{equation*}
	(a,t)(b,u) = ((\ran(tu)\cdot a)(t\cdot b),tu).
\end{equation*}
The second projection $\pi_2\colon \lsd KT \to T,\ (a,t)\mapsto t$ is obviously a surjective homomorphism. 
A routine calculation shows that the Kernel of $\pi_2$ is
$\bbK = \bigcup_{e\in E(T)} \bbK_e$ where
\begin{equation}
\label{equ:Ker-lsd}
\bbK_e = \{(a,e)\in K\times E(T): e\cdot a = a\} = e\cdot K \times \{e\}\ 
(e\in E(T)),
\end{equation}
and so $\lsd KT$ is a normal extension of $\bbK$ by $T$ along 
$\eta\colon \bbK\to E(T),\ (a,e) \mapsto e$.
Notice that $\lsd KT = \lsd {K'}T$ where 
$K' = \{a\in K : a = e\cdot a\ \mbox{for some}\ e\in E(T)\} = 
\bigcup \{e\cdot K : e\in E(T)\}$ 
is an inverse subsemigroup in $K$.
Consequently, we can suppose without loss of generality in every $\lambda$-semidirect product $\lsd KT$ that 
$K = \bigcup \{e\cdot K : e\in E(T)\}$.

Now assume that the action of $T$ on $K$ has the following property:

\medskip\noindent
	(\hbox{AFR})\quad 
	there exists a surjective homomorphism $\epsilon\colon K\to E(T)$ such that
	\begin{equation}\label{conAFR}
	e\cdot a = a\quad \hbox{if and only if}\quad \epsilon(a) \le e,\quad \hbox{for all} \ a\in K\ \hbox{and}\ e\in E(T).
	\end{equation}
\smallskip\noindent
Then 
\begin{equation*}
\rsd KT = \{(a,t)\in K\times T: \epsilon(a) = \ran(t)\}
\end{equation*}
forms an inverse subsemigroup in $\lsd KT$ where the operation has the form 
\begin{equation*}
(a,t)(b,u) = (a(t\cdot b),tu)
\end{equation*}
usual in semidirect products of groups and semigroups.
The inverse semigroup $\rsd KT$ is called the \emph{full restricted semidirect product of $K$ by $T$ with respect to the given action (having property\/ {\rm(AFR)})}.  
The second projection $\pi_2\colon \rsd KT\to T$ is a surjective homomorphism also in this case, and its Kernel is easily seen to be 
$\bbK = \bigcup_{e\in E(T)} \bbK_e$ where
\begin{equation}
\label{equ:Ker-rsd}
\bbK_e = \{(a,e)\in K\times E(T): \epsilon(a) = e\} = K_e \times \{e\}\ (e\in E(T)).
\end{equation}
Consequently, $\bbK$ is isomorphic to $K$, and 
$\rsd KT$ is a normal extension of $K$ by $T$ along $\epsilon$.

If $\gsd KT$ is a $\lambda$-semidirect or a full restricted semidirect product of $K$ by $T$ then the only congruence considered in the paper on it will be 
$\ker \pi_2$.
Therefore it causes no confusion if we denote the normal extension $(\gsd KT,\ker \pi_2)$ simply by $\gsd KT$.
This normal extension corresponding to a $\lambda$-semidirect or a full restricted semidirect product is meant also when saying, for instance, that a normal extension is embeddable in (resp. isomorphic to) a $\lambda$-semidirect or a full restricted semidirect product.

Comparing these constructions to each other as normal extensions, a great disadvantage of $\lambda$-semidirect product is that its Kernel is far from being isomorphic to the first factor in general.

An important connection between these two constructions is noticed in \cite[Lemma 3]{BKK}.
Without mentioning that $\bbK$ in (\ref{equ:Ker-lsd}) is the Kernel of the second projection $\pi_2$ of $\lsd KT$, an action of $T$ on $\bbK$ is defined by means of the action of $T$ on $K$ by the rule 
\begin{equation} \label{equ:act-on-bbK}
t\cdot (a,e) = (t\cdot a,\ran(te))\quad (t\in T,\ (a,e)\in \bbK_e),
\end{equation}
and it is checked that the homomorphism $\epsilon\colon \bbK\to E(T)$ corresponding to the decomposition $\bbK = \bigcup_{e\in E(T)} \bbK_e$ in
(\ref{equ:Ker-lsd}) is appropriate for defining a full restricted semidirect product $\rsd{\bbK}{T}$.
Moreover, it is shown that the function 
\begin{equation*}
\psi\colon \lsd KT \to \rsd{\bbK} {T},\quad 
   \psi(a,t) = ((a,\ran(t)),t)
\end{equation*}
is an injective homomorphism.
However, this can be easily strengthened as follows.

\bigskip

\begin{Lem}
	The function $\psi$ is 
	an isomorphism of normal extensions. 
\end{Lem}

\begin{proof}
	It is straightforward that 
	$\psi(\bbK)$ coincides with the Kernel of the second projection 
	$\pi_2^{\rsd{\bbK}{T}}$ 
	of $\rsd{\bbK} {T}$, and that
	$(a,t) \mathrel{\ker \pi_2} (a',t')$ for some $(a,t),(a',t')\in \lsd KT$ if and only if $((a,\ran(t)),t) \mathrel{\ker \pi_2^{\rsd{\bbK}{T}}} ((a',\ran(t')),t')$ in $\rsd{\bbK} {T}$.
	Thus $\psi$ is an embedding from the normal extensions $\lsd KT$ to the normal extension $\rsd{\bbK} {T}$, and in order to establish that it is also an isomorphism, it suffices to see that $\psi$ is also surjective.
	Let $((a,e),t)$ be an arbitrary element of $\rsd{\bbK} {T}$.
	Then $e = \epsilon(a,e) = \ran(t)$ by the definition of $\rsd{\bbK}{T}$, and
	$e\cdot a = a$ by the definition of $\bbK$. 
	Hence $((a,e),t) = \psi(a,t)$ and surjectivity of $\psi$ is also verified. 
\end{proof}

This immediately implies the following.

\bigskip

\begin{Prop}
	Let $(K,\eta,T)$ be a normal extension triple.
	A normal extension of $K$ by $T$ along $\eta$ is embeddable, as a normal extension, in a $\lambda$-semi\-direct product $\lsd{\olK}{T}$ for some inverse semigroup $\olK$ if and only if it is embeddable in the full restricted semidirect product of \,$\overline{\mathbb{K}}$ by $T$ where $\overline{\mathbb{K}}$ is the Kernel of the second projection of \,$\lsd{\olK}{T}$.
\end{Prop}

\bigskip

\begin{Rem} \label{rem:act-on-bbK}
	{\rm 
	For our later convenience, notice that the action of $T$ on $\bbK$ induced by the isomorphism 
	$K\to \bbK,\ a\mapsto (a,e)$ from the first factor of a full restricted semidirect product $\rsd KT$ to the Kernel of the second projection of $\rsd KT$ in (\ref{equ:Ker-rsd}) is just that in (\ref{equ:act-on-bbK}).
	}
\end{Rem}

Let $K$ and $T$ be inverse semigroups, and denote by $P_{K,T}$
the set $\bigcup_{e\in E(T)} K^{Te}$ 
of all functions from principal left ideals of $T$ into $K$.
The domain of a function $\alpha\in P_{K,T}$ is denoted by $\Hdom \alpha$.
Define `pointwise' multiplication $\oplus$ on $P_{K,T}$ in the usual way: for any $\alpha,\beta\in P_{K,T}$, let
$\Hdom(\alpha\oplus \beta)=\Hdom \alpha\cap\Hdom \beta$, and 
$(\alpha\oplus \beta)(x)=\alpha(x)\beta(x)$ for every $x\in \Hdom(\alpha\oplus \beta)$.
Since the intersection of two principal left ideals of an inverse semigroup is a principal left ideal,
$P_{K,T}$ forms an inverse semigroup with respect to the operation $\oplus$.
Moreover, introduce an action of $T$ on $P_{K,T}$ by endomorphisms 
as follows: for every $t\in T$ and $\alpha\in P_{K,T}$,
let $t\cdot \alpha\colon (\Hdom \alpha)t^{-1}\to K,\ x\mapsto \alpha(xt)$.
Finally, consider the set
\begin{equation*}
	\Hwr KT=\{(\alpha,t)\in P_{K,T}\times T: \Hdom \alpha = Tt^{-1}\},
\end{equation*}
and define a  multiplication on it by the rule
\begin{equation*}
	(\alpha,t)(\beta,u)=(\alpha\oplus (t\cdot \beta),tu).
\end{equation*}
The inverse semigroup $\Hwr KT$ obtained in this way is called \emph{Houghton's wreath product of $K$ by $T$}.

Notice that if $e,f\in E(T)$ and $\alpha\in P_{K,T}$ such that $\Hdom \alpha = Tf$
then $\Hdom (e\cdot \alpha) = \Hdom \alpha$ if and only if $f \le e$, and this holds if and only if $e\cdot \alpha = \alpha$.
Thus the action of $T$ on $P_{K,T}$ defined above satisfies condition (\ref{conAFR}) for the function $\epsilon\colon P_{K,T}\to E(T)$ where $\epsilon(\alpha)$ is chosen to be the unique idempotent generator of the principal left ideal $\Hdom\alpha$ for any $\alpha \in P_{K,T}$. 
This defines a full restricted semidirect product $\rsd {P_{K,T}}T$, and it is easy to see that $\Hwr KT = \rsd {P_{K,T}}T$. 

\subsection*{Translations}

Let $S$ be a semigroup.
A transformation $\lambda$ of $S$ is a \emph{left translation} if 
$\lambda(st) = (\lambda(s))t$ for every $s,t\in S$, and a transformation $\rho$ of $S$, written as a right operator, is a \emph{right translation} if
$(st)\rho = s((t)\rho)\ (s,t\in S)$.
If $s(\lambda(t)) =((s)\rho)t$ also holds for any $s,t\in S$ then $\lambda$ and $\rho$ are \emph{linked}, and the pair 
$(\lambda,\rho)$
is called a \emph{bitranslation of $S$}.
The set $\Lambda(S)$ (resp.\ $\Rho(S)$) of all left (resp.\ right) translations of $S$ forms a submonoid in the monoid of all transformations (transformations, considered as right operators) of $S$. 
Furthermore, it is easy to verify that the set of all bitranslations of $S$ is a submonoid in the direct product $\Lambda(S)\times \Rho(S)$.
This submonoid is called the \emph{translational hull of $S$} and is denoted by $\Omega(S)$. 
The projections
\begin{equation*}
	\Upsilon_\Lambda\colon \Omega(S)\to \Lambda(S),\ (\lambda,\rho)\mapsto \lambda \quad \hbox{and}\quad  
	\Upsilon_\Rho\colon \Omega(S)\to \Rho(S),\ (\lambda,\rho)\mapsto \rho
\end{equation*}
are obviously homomorphisms. 

To reduce the number of letters and parentheses, we will use bitranslations as  `bioperators'. 
If $\omega\in \Omega(S)$ where $\omega = (\lambda,\rho)$ then we define $\omega s$ to be $\lambda(s)$ and $s\omega$ to be $(s)\rho$.
Thus the equalities in the previous paragraph have the forms
\begin{equation*}
	\omega(st) = (\omega s)t,\quad (st)\omega = s(t\omega)\quad \hbox{and}\quad
	s(\omega t) =(s\omega) t,
\end{equation*}
respectively.

Each element $s$ of $S$ defines a bitranslation $\pi_s$ by  
$\pi_s t = st,\, t\pi_s = ts\  (t\in S)$ which is called the \emph{inner bitranslation induced by} $s$.
Denote the set of all inner bitranslations by $\Pi(S)$.
It is easy to verify that we have $\omega\pi_s = \pi_{\omega s}$ and 
$\pi_s \omega = \pi_{s\omega}$ for every $s\in S$ and $\omega\in \Omega(S)$.  
Consequently, $\Pi(S)$ is an ideal in $\Omega(S)$.
Moreover, the function $\pi\colon S\to \Pi(S),\ s\mapsto \pi_s$ is a homomorphism, called the \emph{canonical homomorphism from $S$ to $\Omega(S)$}.

Now let $S$ be an inverse semigroup.
It is well known that $\Omega(S)$ is an inverse monoid, and the canonical homomorphism $\pi$ is injective, thus implying that $\Pi(S)$ is isomorphic to $S$.
The projections $\Upsilon_\Lambda$ and $\Upsilon_\Rho$ of $\Omega(S)$ into $\Lambda(S)$ and $\Rho(S)$, respectively, are also injective. 
This implies that $\Omega(S)$ is isomorphic to both $\Upsilon_\Lambda(\Omega(S))$ and  $\Upsilon_\Rho(\Omega(S))$.
The following properties will be useful in calculations:
\begin{equation*}
	\omega e = e\omega\in E(S)\quad \hbox{for every}\quad e\in E(S)\ \ \hbox{and}\ \ \omega\in E(\Omega(S))
\end{equation*}
and
\begin{equation*}
	(\omega a)^{-1} = a^{-1}\omega^{-1}\quad \hbox{for every}\quad a\in S\ \ \hbox{and}\ \ \omega\in \Omega(S).
\end{equation*}

\section
{Abstract characterization of full restricted semidirect products}
\label{sect:char-rsd}

The aim of this section is to describe, up to isomorphism, the full restricted semidirect products as being just the normal extensions which are defined by a class of congruences generalizing split Billhardt congruences.

Before turning to the main point of this section, 
we give an alternative description for 
the actions having property (AFR) which allows us to simplify later arguments.

\bigskip

\begin{Prop}\label{prop:rsd-act}
	Suppose that $K$ and  $T$ are inverse semigroups and $T$ acts on $K$ by endomorphisms.
	Let $\epsilon\colon K\to E(T)$ be an arbitrary surjective homomorphism.
	Then the action of $T$ on $K$ and the homomorphism $\epsilon$ 
	satisfy condition (\ref{conAFR}) if and only if 
	the following properties hold:
	\begin{equation}\label{AE7}
	\epsilon(a)\cdot a = a\quad \hbox{for every\/}\ a\in K,
	\end{equation} 
	\begin{equation}\label{AE8}
	\epsilon(t\cdot a) = \ran(t\epsilon(a))\quad \hbox{for every\/}\ a\in K\ \hbox{and\/}\ t\in T.
	\end{equation} 
\end{Prop}

\begin{proof}
	First assume that (\ref{AE7}) and (\ref{AE8}) are fulfilled, and let $e\in E(T)$ and $a\in K$.
	If $e\cdot a = a$ then (\ref{AE8}) implies that 
	$\epsilon(a) = \epsilon(e\cdot a) = e\epsilon(a)$
	whence $\epsilon(a)\le e$.
	Conversely, if  $\epsilon(a)\le e$ then it follows by (\ref{AE7}) that
	\begin{equation*}
		a = \epsilon(a)\cdot a = (e\epsilon(a))\cdot a = e\cdot (\epsilon(a)\cdot a) = e\cdot a.
	\end{equation*}
	Thus we have shown that (\ref{conAFR}) holds.
	
	Now suppose that (\ref{conAFR}) is satisfied.
	Clearly, (\ref{conAFR}) implies (\ref{AE7}), thus, in order to prove the `only if' part of the statement, it suffices to check that property (\ref{AE8}) also holds.
	We start the argument with verifying the equalities
	\begin{equation}
	\label{equ:rsd-act}
	\epsilon(\epsilon(a)\cdot b) = \epsilon(ab) = \epsilon(\epsilon(b)\cdot a)
	\quad \hbox{for every}\quad a,b\in K.
	\end{equation}
	Consider arbitrary elements $a,b\in K$.
	Then we see by (\ref{conAFR}) that
	\begin{eqnarray*}
		ab & = & \epsilon(ab)\cdot ab = (\epsilon(a)\epsilon(b))\cdot ab \cr
		   & = & \big(\epsilon(a)\epsilon(b)\cdot a\big)
		\big(\epsilon(a)\epsilon(b)\cdot b\big) = (\epsilon(b)\cdot a)(\epsilon(a)\cdot b)
	\end{eqnarray*}
	which implies that
	\begin{equation*}
		\epsilon(ab)
		= \epsilon(\epsilon(b)\cdot a)\epsilon(\epsilon(a)\cdot b)
		\le \epsilon(\epsilon(b)\cdot a),\ \epsilon(\epsilon(a)\cdot b).
	\end{equation*}
	On the other hand, we have
	\begin{equation*}
	\epsilon(b)\cdot a = \epsilon(b)\cdot (\epsilon(a)\cdot a) 
	= \epsilon(a)\epsilon(b)\cdot a 
	= (\epsilon(a)\epsilon(b))\epsilon(b) \cdot a 
	= \epsilon(ab)\cdot (\epsilon(b)\cdot a) 
	\end{equation*}
	whence we obtain $\epsilon(ab) \ge \epsilon(\epsilon(b)\cdot a)$ by (\ref{conAFR}), and the inequality
	$\epsilon(ab) \ge \epsilon(\epsilon(a)\cdot b)$ is seen in a similar way.
	This verifies (\ref{equ:rsd-act}).
	Since $\epsilon$ is surjective, for any $e\in E(T)$, we have $b\in K$ such that $e = \epsilon(b)$.
	Applying (\ref{equ:rsd-act}), we obtain that $\epsilon(e\cdot a) = \epsilon(\epsilon(b)\cdot a) = \epsilon(b)\epsilon(a) = e\epsilon(a)$,
	that is, we have
	\begin{equation}
	\label{equ:rsd-actspec}
	\epsilon(e\cdot a) = e\epsilon(a)\quad \hbox{for every}\quad 
	a\in K,\ e\in E(T).
	\end{equation}
	If $a\in K$ and $t\in T$ are arbitrary elements then the equality
	$t\cdot a = t\cdot a'$ is valid for the element $a'= \dom(t)\cdot a$ whence
	$\dom(t)\cdot a' = a'$ is clear, and we have
	$\epsilon(a') = \dom(t) \epsilon(a)$ by (\ref{equ:rsd-actspec}).
	Thus \cite[Lemma 5.3.8(1)]{LawInvSg} implies that
	\begin{equation*}
		\epsilon(t\cdot a) = \epsilon(t\cdot a') = \ran(t\epsilon(a')) = 
		\ran(t\dom(t)\epsilon(a)) = \ran(t\epsilon(a)),
	\end{equation*}
	and this completes the proof of (\ref{AE8}).
\end{proof}

\smallskip

\begin{Rem}
	{\rm 
	Notice that  (\ref{AE8}) extends property \cite[Lemma 5.3.8(1)]{LawInvSg} to arbitrary elements of $K$ and $T$,
	and by making use of (\ref{AE7}) and (\ref{AE8}), the proof of 
	\cite[Theorem 5.3.5]{LawInvSg} can be simplified.
	}		
\end{Rem}

\bigskip

\begin{Rem} \label{rem:conAE78}
	{\rm 
	Suppose that $K$, $T$ and $\epsilon$ satisfy the assumptions of Proposition \ref{prop:rsd-act}.
	If the semilattice decomposition defined by $\epsilon$ is 
	$K = \bigcup_{e\in E(T)} K_e$ then properties (\ref{AE7}) and (\ref{AE8}) are equivalent to the following two}
	\begin{eqnarray} 
	\label{conAE7mod}
	&\quad & e\cdot a = a\quad \hbox{\rm for every}\ e\in E(T)\ \hbox{\rm and}\ a\in K_e \\
	\label{conAE8mod}
	&\quad & t\cdot a \in K_{\ran(te)}\quad \hbox{\rm for every}\ e\in E(T),\ a\in K_e\  \hbox{\rm and}\ t\in T,
	\end{eqnarray} 
	{\rm respectively.
	Moreover, the underlying set of the respective full restricted semidirect product is}
	\begin{equation*}
	\rsd KT = \{(a,t)\in K\times T: a\in K_{\ran(t)}\}.
	\end{equation*}
\end{Rem}

Consider a surjective homomorphism $\epsilon\colon K\to E$ from an inverse semigroup $K$ to a semilattice $E$, and let the respective semilattice decomposition be $K = \bigcup_{e\in E} K_e$.
It is routine to verify that if $E$ acts on $K$ by endomorphisms such $\epsilon$ satisfies (\ref{conAFR}) then $K$ is a strong semilattice $E$ of its inverse subsemigroups $K_e\ (e\in E)$ with structure homomorphisms 
$\varepsilon_{e,f}\colon K_e\to K_f\ (f\le e)$
which are given by the equalities
\begin{equation*}
	f\cdot a = \varepsilon_{e,f}(a)\quad (e,f\in E\ \hbox{with}\ e\ge f\ \hbox{and}\ a\in K_e).
\end{equation*}
This implies the important consequence of Proposition \ref{prop:rsd-act}, more precisely, of (\ref{equ:rsd-actspec}), formulated in Corollary \ref{cor:act-strongsl}.
It is also worth noticing that, conversely, if $K$ is a strong semilattice of its inverse subsemigroups $K_e\ (e\in E)$ then $E$ acts on $K$ by endomorphisms  in a way that condition (\ref{conAFR}) is fulfilled by the homomorphism $\epsilon\colon K\to E$ corresponding to the semilattice decomposition of $K$.
Namely, if the family of structure homomorphisms in $K$ is 
$\varepsilon_{e,f}\ (e,f\in E,\ f\le e)$ then the appropriate action is defined as follows:
$f\cdot a = \varepsilon_{f,ef}(a)$ for every $e,f\in E$ and $a\in K_e$.

\bigskip

\begin{Cor}\label{cor:act-strongsl}
	If $K$ and  $T$ are inverse semigroups and $T$ acts on $K$ by endomorphisms such that axiom {\rm (AFR)} holds then $K$ is a strong semilattice of its inverse subsemigroups $K_e\ (e\in E(T))$.
\end{Cor}

\medskip

Now we turn our attention to the main objective of the section.
First we introduce the concept of the translational hull of a normal extension and several notions and notation related to it.

Let $(S,\theta)$ be a normal extension and let $K = \Ker \theta$.
We say that a bitranslation \emph{$\omega\in \Omega(S)$ respects the congruence\/ $\theta$} if we have 
$\omega s\rtheta \omega s'$ and $s\omega\rtheta s'\omega$ 
for every $s,s'\in S$ with $s\rtheta s'$.
It is routine to check that the set of all bitranslations of $S$ respecting $\theta$ forms an inverse submonoid $\Omega_\theta(S)$ in $\Omega(S)$.
It is worth mentioning that $\Pi(S)$ is contained in $\Omega_\theta(S)$, and since $\Pi(S)$ is an ideal in $\Omega(S)$, it is an ideal also in $\Omega_\theta(S)$.

Notice that each $\omega\in \Omega_\theta(S)$ induces a bitranslation 
$\omega^\downharpoonright$ 
on the factor semigroup $S/\theta$ in a natural way:
\begin{equation}
\label{equ:harp}
\omega^\downharpoonright(\theta(s)) = \theta(\omega s) \quad
	\hbox{and} \quad
(\theta(s))\omega^\downharpoonright = \theta(s\omega) \quad
	\hbox{for any}\ s\in S.
\end{equation}
It is easy to verify that the function 
$()^\downharpoonright\colon \Omega_\theta(S)\to \Omega(S/\theta),\ 
\omega\mapsto \omega^\downharpoonright$ is a homomorphism.
Actually, $\Omega_\theta(S)$ consists just of the bitranslations of $S$ for which rule (\ref{equ:harp}) defines a bitranslation of $S/\theta$.
In particular, since $()^\downharpoonright$ is a homomorphism,
\begin{equation*}
\Omega(S,\theta) = \{\omega\in \Omega_\theta(S): 
\omega^\downharpoonright\in \Pi(S/\theta)\}
\end{equation*}
is an inverse subsemigroup of $\Omega_\theta(S)$ containing $\Pi(S)$ as an ideal.
We call $\Omega(S,\theta)$ the \emph{translational hull of the normal extension $(S,\theta)$}.
For simplicity, the restriction of the homomorphism $()^\downharpoonright$ to $\Omega(S,\theta)$ is also denoted by $()^\downharpoonright$.
The congruence on $\Omega(S,\theta)$ induced by 
$()^\downharpoonright$ is the relation $\Omega(\theta)$ given in the following way:
\begin{eqnarray}
	\label{Omtheta}
	\lefteqn{\omega \mathrel{\Omega(\theta)} \omega' \quad 
	\hbox{if and only if}} \cr 
	& & \qquad\quad\ \; \omega s \mathrel{\theta} \omega' s \ \  \hbox{and} 
	\ \  s \omega \mathrel{\theta} s \omega' \quad 
	\hbox{for every} \ \  s\in S \ \  (\omega, \omega'\in \Omega(S,\theta)).
\end{eqnarray}
Its restriction to $\Pi(S)$ is denoted by $\Pi(\theta)$.
The properties of the translational hull of a normal extension are summarized in the following proposition.

\bigskip

\begin{Prop} \label{prop:trhull-next}
	Let $(S,\theta)$ be a normal extension.
	\begin{enumerate}
		\item 
		\label{prop:trhull-next1}
		The function
		$()^\downharpoonright\colon \Omega(S,\theta)\to \Pi(S/\theta),\ 
		\omega\mapsto \omega^\downharpoonright$
		defined by (\ref{equ:harp}) is a surjective homomorphism, and its kernel is the congruence $\Omega(\theta)$ given in (\ref{Omtheta}). 
		\item 
		\label{prop:trhull-next2}
		The canonical homomorphism $\pi \colon S\to \Omega(S)$ embeds the normal extension $(S,\theta)$ into the normal extension $(\Omega(S,\theta),\Omega(\theta))$ such that
		\begin{equation*}
			\iota\colon S/\theta \to \Omega(S,\theta)/\Omega(\theta),\ 
			\theta(s) \mapsto \Omega(\theta)(\pi_s)
		\end{equation*}
		is an isomorphism.
	\end{enumerate}
\end{Prop}

Let $K,T$ be inverse semigroups, and let $T$ act on $K$ such that (AFR) is satisfied. 
Consider the full restricted semidirect product $\bbS = \rsd KT$ defined by them.
For every $t\in T$, let us introduce a bioperator $\omega_{[t]}$ on $\bbS$ as follows: 
\begin{equation*}
\omega_{[t]}(x,u) = (t\cdot x,tu)\quad \hbox{and} \quad
(x,u)\omega_{[t]} = (\ran(ut)\cdot x,ut)\quad
((x,u)\in \bbS).
\end{equation*}
First of all, we establish that $\omega_{[t]}\in \Omega(\bbS,\Theta)$ where $\Theta$ is the congruence induced on $\bbS$ by  the second projection.
Since $\bbS/\Theta$ is isomorphic to $T$ we consider the homomorphism $()^\downharpoonright$ to be a function into $\Omega(T)$ rather than into $\Omega(\bbS/\Theta)$,
and so bitranslations of $T$ will occur in the arguments.
To avoid confusion we will distinguish them from bitranslations of $\bbS$ by means of a superscript $T$.

\bigskip

\begin{Lem}
\label{lem:lambda-t-trans}
For any $t\in T$, we have $\omega_{[t]}\in \Omega(\bbS,\Theta)$ where $\omega^\downharpoonright_{[t]}=\omega^T_t$.
\end{Lem}

\begin{proof}
	It suffices to verify that $\omega_{[t]}\in \Omega(\bbS)$ since, by definition, the second components of $\omega_{[t]}(x,u)$ 
	and $(x,u)\omega_{[t]}$ are $tu = \omega^T_t u$ and $ut = u\omega^T_t$, respectively, for every $x\in K_{\ran(u)}$ whence
	$\omega_{[t]}\in \Omega(\bbS,\Theta)$ and
	$\omega^\downharpoonright_{[t]}=\omega^T_t$ follow.
	Let $(x,u), (y,v)\in \bbS$ be arbitrary elements. 
	Then
	\begin{eqnarray*}
	\omega_{[t]}\left((x,u)(y,v)\right) 
	& = & \omega_{[t]}\left(x(u\cdot y),uv\right)
	= \left(t\cdot (x(u\cdot y)),tuv\right) \cr
	& = & \left((t\cdot x)(tu\cdot y),tuv\right) 
	= (t\cdot x,tu)(y,v)
	= \left(\omega_{[t]}(x,u)\right)(y,v)
	\end{eqnarray*}
	whence we see that $\omega_{[t]}$, as a left operator, is a left translation of $\bbS$.
	Similarly, we obtain that
	\begin{eqnarray*}
	\left((x,u)(y,v)\right)\omega_{[t]}
	& = & \left(x(u\cdot y),uv)\right)\omega_{[t]}
	= \left(\ran(uvt)\cdot (x(u\cdot y)),uvt\right) \cr
	& = & \left(x(\ran(uvt)u\cdot y),uvt\right)
	= \left(x(u\ran(vt)\cdot y),uvt\right) \cr
	& = & (x,u)\left(\ran(vt)\cdot y,vt\right)
	= (x,u)\left((y,v)\omega_{[t]}\right)
	\end{eqnarray*}
	where the third equality is implied by the facts that $e = \ran(uvt)\in E(T)$
	and $e\cdot (x(u\cdot y)) = (e\cdot x)(e^2u\cdot y) = e\cdot (x(eu\cdot y)) =
	x(eu\cdot y)$ since $x(eu\cdot y)\in K_{\ran(u)}K_{\ran(euv)}\subseteq K_e$. 
	Thus $\omega_{[t]}$, as a right operator, is a right translation. 
	Finally, we verify that they are linked. 
	Indeed,
	\begin{eqnarray*}
	\left((x,u)\omega_{[t]}\right)(y,v)
	& = & (\ran(ut)\cdot x,ut)(y,v)
	= ((\ran(ut)\cdot x)(ut\cdot y),utv) \cr
	& = & \left(\ran(ut)\cdot(x(ut\cdot y)),utv\right)
	= (x(ut\cdot y),utv)
	\end{eqnarray*}
	since $x(ut\cdot y)\in K_{\ran(u)}K_{\ran(utv)}\subseteq K_{\ran(utv)}$ where $\ran(utv)\le \ran(ut)$, and
	\begin{equation*}
	(x(ut\cdot y),utv) = (x,u)(t\cdot y,tv)
	= (x,u)\left(\omega_{[t]}(y,v)\right).
	\end{equation*}
	\end{proof}

Let us introduce the function 
$\omega_{[\;]}\colon T\to \Omega(\bbS,\Theta),\ t\mapsto \omega_{[t]}$.

\bigskip

\begin{Lem}
\label{lem:hom-lambda}
	The function $\omega_{[\;]}$ is an injective homomorphism such that 
	$\omega_{[t]} \mathrel{\Omega(\Theta)} \omega_{(a,t)}$ for every
	$(a,t)\in \bbS$.
\end{Lem}

\begin{proof}
	Let $t,u\in T$ and $(x,v)\in \bbS$.
	It is straightforward by definition that
	$\omega_{[t]}\omega_{[u]}(x,v) = \omega_{[tu]}(x,v)$ and 
	$(x,v)\omega_{[t]}\omega_{[u]} = (x,v)\omega_{[tu]}$
	whence $\omega_{[\;]}$ is, indeed, a homomorphism.
	Since $\omega^\downharpoonright_{(a,t)} = \omega^T_t$ for every $t\in T$,
	Lemma \ref{lem:lambda-t-trans} implies that $\omega_{[\;]}$ is injective, and
	$\omega_{(a,t)} \mathrel{\Omega(\Theta)} \omega_{[t]}$ for all 
	$(a,t)\in \bbS$.
\end{proof}

In the next lemma we formulate further properties of $\omega_{[t]}\ (t\in T)$.
Recall (\ref{equ:Ker-rsd}) and Remark \ref{rem:act-on-bbK} for the Kernel $\bbK$ of $\Theta$ and for the action of $T$ induced on $\bbK$, respectively.

\bigskip

\begin{Lem}
\label{lem:lambda-t}
For every elements $t\in T$, $e\in E(T)$, $(a,t)\in \bbS$ and $(c,e)\in \bbK_e$, we have
	\begin{enumerate}
		\item 
		\label{lem:lambda-t:np-order}
			$\omega^{-1}_{[t]}\omega_{[t]} = \omega_{[\dom(t)]} \ge \omega_{(t^{-1}\cdot \dom(a),\dom(t))} = \omega_{\dom(a,t)} = 
			\omega^{-1}_{(a,t)}\omega_{(a,t)}$,
		\item 
		\label{lem:lambda-t:act-vs-trans}
			$t\cdot (c,e) = \omega_{[t]}(c,e)\omega^{-1}_{[t]}$.
	\end{enumerate}
\end{Lem}

\begin{proof}
	(\ref{lem:lambda-t:np-order})\quad
	The equality relations clearly follow by Lemma \ref{lem:hom-lambda} and by definitions.
	Thus in order to prove the inequality, it suffices to verify that if 
	$(i,e)\in E(\bbS)$, that is, if $e\in E(T)$ and $i\in E(K_e)$ then 
	$\omega_{[e]} \ge \omega_{(i,e)}$.
	Indeed, definitions easily imply for any $(x,u)\in \bbS$ that 
	$$\omega_{(i,e)}(x,u) = (i(e\cdot x),eu) \le (e\cdot x,eu) = \omega_{[e]}(x,u)$$
	and
	\begin{eqnarray*}
	(x,u)\omega_{(i,e)} 
	   & = & (x(u\cdot i),ue) 
	   = \big(\!\ran(ue)\cdot(x(ue\cdot i)),ue\big) \cr
	   & = & \big((\ran(ue)\cdot x)((ue\cdot i)),ue\big)
	   \le (\ran(ue)\cdot x,ue) = (x,u)\omega_{[e]}.
	\end{eqnarray*}

	(\ref{lem:lambda-t:act-vs-trans})\quad
	The equality is routine to check by definitions:
	\begin{eqnarray*}
	\omega_{[t]}(c,e)\omega^{-1}_{[t]} 
	 & = & \omega_{[t]}(c,e)\omega_{[t^{-1}]}
	 = \omega_{[t]}(\ran(et^{-1})\cdot c,et^{-1}) 
	 = (t\ran(et^{-1})\cdot c,tet^{-1})\cr
	 & = & (\ran(te)t\cdot c,\ran(te)) = (t\cdot c,\ran(te))
	 = t\cdot (c,e).
	\end{eqnarray*}
\end{proof}

Consider the subset 
$\olbbS=\Pi(\bbS)\cup \omega_{[\;]}(T)$ of
$\Omega(\bbS,\Theta)$.
Since $\Pi(\bbS)$ is an ideal in $\Omega(\bbS,\Theta)$ it follows 
by Lemma \ref{lem:hom-lambda} that $\olbbS$ is an inverse subsemigroup of $\Omega(\bbS,\Theta)$ in which $\Pi(\bbS)$ is an ideal isomorphic to $\bbS$
and $\omega_{[\;]}(T)$ is a subsemigroup isomorphic to $T$.
Moreover, Lemmas \ref{lem:hom-lambda} and \ref{lem:lambda-t}(\ref{lem:lambda-t:np-order}) imply that the restriction $\olTheta$ of $\Omega(\Theta)$ to $\olbbS$ is a split Billhardt congruence on $\olbbS$.
This motivates the following definitions.

Let $(S,\theta)$ be a normal extension and let 
$\xi\colon S/\theta \to \Omega(S,\theta)$ be a function.
Consider the inverse subsemigroup $\olS$ of $\Omega(S,\theta)$ generated by 
$\Pi(S) \cup \xi(S/\theta)$.
For simplicity, we write $\xi^{-1}(t)$ for the inverse of an element 
$\xi(t)\in \Omega(S,\theta)\ (t\in S/\theta)$.
We say that the function $\xi$ is an \emph{almost Billhardt transversal to} $\theta$ if the following conditions are satisfied:

\medskip\noindent
	(B1)\quad 
	$(\xi(t))^\downharpoonright = \omega^{S/\theta}_t$
	for every $t\in S/\theta$,

\smallskip\noindent
	(B2)\quad 
	$\xi^{-1}(t) \xi(t) \ge \omega^{-1} \omega$
	for every element $\omega\in \olS\setminus \xi(S/\theta)$ such that 
	$\omega^\downharpoonright = \omega^{S/\theta}_t$. 

\medskip\noindent
We call $\theta$ an \emph{almost Billhardt congruence} on $S$ if
there exists an almost Billhardt transversal to $\theta$.
For our later convenience, notice that (B1) implies by 
Proposition \ref{prop:trhull-next}(\ref{prop:trhull-next1}) that
\begin{equation*}
(\xi^{-1}(t))^\downharpoonright = \omega^{S/\theta}_{t^{-1}}
\quad (t\in S/\theta)
\end{equation*}
for every almost Billhardt transversal $\xi$ to $\theta$. 

If $\xi$ is an almost Billhardt transversal to $\theta$ which is also a homomorphism then $\theta$ is called a \emph{split almost Billhardt transversal to} $\theta$.
We say that $\theta$ is a \emph{split almost Billhardt congruence} on $S$ if there exists a split almost Billhardt transversal to $\theta$.
In this case, the inverse subsemigroup $\olS$ of $\Omega(S,\theta)$ generated by $\Pi(S) \cup \xi(S/\theta)$ coincides with $\Pi(S) \cup \xi(S/\theta)$ since $\xi(S/\theta)$ is an inverse subsemigroup and $\Pi(S)$ is an ideal in $\Omega(S,\theta)$.
Consequently, property (B2) is equivalent in this special case to 

\medskip\noindent
	(sB2)\quad 
	$\xi(\dom(t)) \ge \omega_{\dom(s)}$
	for every $s\in S$ such that $\theta(s) = t$. 

\medskip
It is straightforward by definition and by Proposition \ref{prop:trhull-next} that if $\theta$ is an almost Billhardt (resp.\ split almost Billhardt) congruence on $S$ then the restriction $\oltheta$ of $\Omega(\theta)$ to $\olS$ is a Billhardt (resp.\ split Billhardt) congruence on $\olS$, and
$\olS/\oltheta$ is isomorphic to $S/\theta$. 
Moreover, the types of congruences just introduced generalize Billhardt and split Billhardt congruences, respectively, since a congruence $\theta$ is a Billhardt (resp.\ split Billhardt) congruence if and only if there exists an almost Billhardt (resp.\ split almost Billhardt) transversal to $\theta$ such that $\xi(S/\theta)\subseteq \Pi(S)$.
This justifies the `only if' parts of the following alternative characterizations of such congruences.

\bigskip

\begin{Prop} 
	Suppose that $S$ is an inverse semigroup and $\theta$ is a congruence on $S$.
	The congruence $\theta$ is an almost Billhardt (resp.\ split almost Billhardt) congruence on $S$ if and only if there exists an inverse subsemigroup $\widetilde{S}$ of $\Omega(S,\theta)$ containing $\Pi(S)$ such that the restriction of $\Omega(\theta)$ to $\widetilde{S}$ is a Billhardt (resp.\ split Billhardt) congruence on $\widetilde{S}$.
\end{Prop}

\begin{proof}
	To show the `if' parts, let $\widetilde{S}$ be a subsemigroup of $\Omega(S,\theta)$ containing $\Pi(S)$, and denote  the restriction of $\Omega(\theta)$ to $\widetilde{S}$ by $\widetilde{\theta}$.
	Suppose that $\widetilde{\xi}\colon \widetilde{S}/\widetilde{\theta}\to \widetilde{S}$ is a Billhardt transversal to $\widetilde{\theta}$.
	Since $\iota$ in Proposition \ref{prop:trhull-next}(\ref{prop:trhull-next2}) 
	is an isomorphism, the function  
	$\xi\colon S/\theta\to \widetilde{S},\ 
	\xi(\theta(s)) = \widetilde{\xi}(\widetilde{\theta}(\pi_s))$
	is an almost Billhardt transversal to $\theta$. 
	In particular, if $\widetilde{\xi}$ is split then $\xi$ is also split.
\end{proof}   

Now we are ready to prove the main result of this section.

\bigskip

\begin{Thm}
	A normal extension $(S,\theta)$ is isomorphic to a full restricted semidirect product if and only if $\theta$ is a split almost Billhardt congruence on $S$. 
\end{Thm}

\begin{proof}
	Since the argument in the paragraph after the proof of Lemma \ref{lem:lambda-t} shows the `only if' part, it suffices to prove the `if' part.
	Suppose that $\theta$ is a split almost Billhardt congruence on $S$.
	For brevity, denote $S/\theta$ by $T$, and let
	$\xi\colon T \to \Omega(S,\theta)$ be a homomorphism such that
	(B1) and (sB2) hold.
	Clearly, the restriction $\xi_0\colon T\to \xi(T),\ t\mapsto \xi(t)$ of $\xi$ is an isomorphism. 
	By the former observations on split almost Billhardt congruences, the subset 
	$\olS = \Pi(S)\cup \xi(T)$ forms an inverse subsemigroup of $\Omega(S,\theta)$ such that the restriction $\oltheta$ of $\Omega(\theta)$ to $\olS$ is a split Billhardt congruence.
	It is easy to see that $\Ker \oltheta =  \Pi(K)\cup \{\xi(\dom(t)): t\in T\} =
	\Pi(K)\cup \xi(E(T))$ which we denote by $\olK$.
	Furthermore, the function $\olxi\colon \olS/\oltheta\to \olS$ defined by 
	$\olxi(\oltheta(\pi_s)) = \xi(\theta(s))$ and 
	$\olxi(\oltheta(\xi(t))) = \xi(t)$ is a split Billhardt transversal to $\oltheta$.
	Hence we obtain by Proposition \ref{prop:trhull-next}(\ref{prop:trhull-next2}) that the function 
	$\iota_0\colon T\to \olS/\oltheta,\ \theta(s)\mapsto \oltheta(\pi_s)$
	is an isomorphism. 
	In order to make the rest of the argument easier to follow we identify $\xi(T)$ and $\olS/\oltheta$ with $T$ via $\xi_0$ and $\iota_0$, respectively.
	
	By applying the proof of \cite[Theorem 5.3.12]{LawInvSg} for $\olS$ and $\oltheta$, we obtain a full restricted semidirect product $\rsd{\olK}{T}$ of $\olK$ by $T$  with respect to the 
	action of $T$ on $\olK$ defined by the rule
	\begin{equation}
	\label{restr:act}
	t\cdot \omega = \xi(t) \omega \xi^{-1}(t)\quad (t\in T,\ \omega\in \olK)
	\end{equation}
	which satisfies condition (\ref{conAFR}) for the surjective homomorphism  
	$\olepsilon\colon \olK \to E(T)$ defined as follows:
	for any $e\in E(T)$, we have 
	$\olepsilon(\omega) = e$ if and only if $\omega = \xi(e)$ or
	$\omega = \omega_a$ for some $a\in K_e$.
	Moreover, we obtain an isomorphism 
	\begin{equation}
	\label{restr:isom}
	\olphi\colon \olS \to \rsd{\olK}{T},\ \ \omega \mapsto (\omega\xi^{-1}(t),t)\quad
	\hbox{if}\ \ \omega\in\Pi(S)\ \hbox{and}\ \omega^\downharpoonright = \omega^T_t,\    
	\hbox{or}\ \ \omega = \xi(t).
	\end{equation}

	Notice that if $\omega\in\Pi(K)$ then we have $t\cdot \omega \in \Pi(K)$ in (\ref{restr:act}), and similarly, if $\omega\in\Pi(S)$ then we have
	$\omega\xi^{-1}(t)\in \Pi(K)$ in (\ref{restr:isom}).
	Thus the action of $T$ on $\olK$ restricts to $\Pi(K)$, and
	the restriction $\epsilon$ of $\olepsilon$ to $\Pi(K)$ has property (\ref{conAFR}).
	Hence the latter action defines a full restricted semidirect product $\rsd{\Pi(K)}{T}$, and the restriction $\phi$ of $\olphi$ to
	$\rsd{\Pi(K)}{T}$ is an isomorphism from $\Pi(S)$ onto
	$\rsd{\Pi(K)}{T}$.
	Consequently $S$ is isomorphic to a full restricted semidirect product of $K$ by $T$, and the theorem is proved.
\end{proof}

\section
{Normal extensions embeddable in full restricted semidirect products}
\label{sect:emb-rsd}

In this section we prove that each normal extension defined by an almost Billhardt congruence is embeddable in a full restricted semidirect product whose Kernel classes are direct products of Kernel classes of the normal extension.
The full restricted semidirect product appearing in the proof is an inverse subsemigroup of Houghton's wreath product of the Kernel of the normal extension by its factor which corresponds to the respective normal extension triple rather than merely to the Kernel and the factor.

Consider a normal extension triple $(K,\eta,T)$. 
For brevity, put $E=E(T)$, and let $K = \bigcup_{e\in E} K_e$ be the semilattice decomposition of $K$ corresponding to $\eta$.
Consider Houghton's wreath product $\Hwr{K}{T} = \rsd {P_{K,T}}T$ of $K$ by $T$, and recall from Section \ref{sect:prelim} that 
$P_{K,T}$ is a semilattice $E$ of the direct powers $K^{Te}\ (e\in E)$.
Define the following subset of the inverse semigroup $P_{K,T}$:
\begin{equation}
	\label{equ:pkteta}
	P^\eta_{K,T} = \{\alpha\in P_{K,T}: \alpha(x)\in K_{\ran(x)}\ \mbox{for every}\ 
	x\in \Hdom \alpha\}.
\end{equation}
Since $K^{Te} = \{\alpha\in P_{K,T}: \Hdom \alpha = Te\}$ the intersection $P^\eta_{K,T}\cap K^{Te}$ is the direct product $P_e = \prod_{x\in Te} K_{\ran(x)}$ for any $e\in E$.
Thus if $e,f\in E$ and $\alpha\in P_e$, $\beta\in P_f$ then
$\alpha^{-1}(x) = (\alpha(x))^{-1}\in K_{\ran(x)}$ for every 
$x\in Te = \Hdom \alpha^{-1}$, and 
$(\alpha\oplus\beta)(x) = \alpha(x)\beta(x)\in K_{\ran(x)}$ for every 
$x\in Tef = \Hdom(\alpha\oplus\beta)$.
Hence $P^\eta_{K,T}$ is an inverse subsemigroup in $P_{K,T}$ which is a semilattice $E$ of the inverse subsemigroups $P_e\ (e\in E)$.
Moreover, $P^\eta_{K,T}$ is closed under the action of $T$ on $P_{K,T}$.
Indeed, if $t\in T$ and $\alpha\in P^\eta_{K,T}$ with $\Hdom \alpha = Te\ (e\in E)$
then $\Hdom(t\cdot \alpha) = Tet^{-1}$, and this implies for every
$x\in \Hdom(t\cdot \alpha)$ that $xtt^{-1} = x$ and
$(t\cdot \alpha)(x) = \alpha(xt) \in K_{\ran(xt)} = K_{\ran(x)}$ 
whence $t\cdot \alpha\in P^\eta_{K,T}$ follows.
In particular, if $t=e\in E$ then we have
$(e\cdot \alpha)(x) = \alpha(xe) = \alpha(x)$ for any 
$x\in \Hdom(e\cdot \alpha)$, and this implies that the action of $T$ on $P^\eta_{K,T}$ induced by the action in the definition of $\Hwr{K}{T}$ has properties (\ref{conAE7mod}) and (\ref{conAE8mod}). 
By Propositon \ref{prop:rsd-act} and Remark \ref{rem:conAE78} this argument verifies the following statement.

\bigskip

\begin{Prop}
	Let $(K,\eta,T)$ be a normal extension triple where the semilattice decomposition of $K$ induced by $\eta$ is $K = \bigcup_{e\in E(T)} K_e$, and consider Houghton's wreath product $\Hwr{K}{T} = \rsd {P_{K,T}}T$ of $K$ by $T$.
	The set $P^\eta_{K,T}$ defined in (\ref{equ:pkteta}) forms an inverse subsemigroup in $P_{K,T}$, and
	the restriction of the action of $T$ on $P_{K,T}$ to 
	$P^\eta_{K,T}$ is an action of $T$ on $P^\eta_{K,T}$ satisfying axiom 
	{\rm (AFR)}.
	Consequently, the full restricted semidirect product 
	$\rsd{P^\eta_{K,T}}{T}$
	of $P^\eta_{K,T}$ by $T$ with respect to this action is an inverse subsemigroup in $\Hwr{K}{T}$.  
\end{Prop}

\medskip

Let us call the inverse semigroup $\rsd{P^\eta_{K,T}}{T}$
\emph{Houghton's wreath product of $K$ by $T$ along $\eta$}, and denote it by
$\Hwrp{K}{T}$.
Now we can formulate our embedding theorem.

\bigskip

\begin{Thm}
	\label{ThmM}
	Let $\theta$ be an almost Billhardt congruence on an inverse semigroup $S$, and let $\eta\colon \Ker \theta \to E(S/\theta)$ be the restriction of $\theta^\natural$.
	Then the normal extension $(S,\theta)$ is embeddable in
	$\Hwrp{\Ker \theta}{S/\theta}$. 
\end{Thm}

\begin{proof}
	Our proof consists of two parts. 
	First we apply the proof of \cite[Theorem 5.3.5]{LawInvSg} (see also \cite{BKK}) for the normal extension $(\olS,\oltheta)$, associated to $(S,\theta)$ in Section 3, to embed $(S,\theta)$ into a $\lambda$-wreath product of $\Ker\theta$ by $S/\theta$. 
	In the second part, we embed the image of this embedding into $\Hwrp{\Ker \theta}{S/\theta}$.
	
	For brevity, denote $\Ker \theta$ by $K$, $S/\theta$ by $T$, $E(T)$ by $E$, and the semilattice decomposition of $K$ corresponding to $\eta$ by
	$K = \bigcup_{e\in E} K_e$.
	Let us fix an almost Billhardt transversal $\xi$ to $\theta$, and denote by $\olS$ 
	the inverse subsemigroup of $\Omega(S,\theta)$ generated by $\Pi(S)\cup \xi(T)$, and by $\oltheta$ the restriction of $\Omega(\theta)$ to $\olS$.
	We have seen in Section \ref{sect:char-rsd} that $\oltheta$ is a Billhardt congruence on $\olS$, and the function $\olxi\colon \olS/\oltheta\to \olS$ defined by 
	$\olxi(\oltheta(\pi_s)) = \xi(\theta(s))$ and 
	$\olxi(\oltheta(\xi(t))) = \xi(t)$ is a Billhardt transversal to $\oltheta$.
	Thus we obtain by the proof of \cite[Theorem 5.3.5]{LawInvSg}
	that, for any $\chi\in \olS$, the function
	\begin{equation}
		\label{equ:olfchi}
		\olf_\chi\colon  \Pi(T)\to \Ker \oltheta,\quad 
		\olf_\chi(\omega^T_t) = \olxi(\omega^T_t (\chi\chi^{-1})^\downharpoonright) \,\chi\, \olxi^{-1}(\omega^T_t \chi^\downharpoonright)
	\end{equation}
	is well defined, and the function
	\begin{equation}
		\label{equ:olphi}
		\olphi\colon \olS \to \lwr{\Ker\oltheta}{\Pi(T)},\quad 
		\olphi(\chi) = (\olf_\chi,\chi^\downharpoonright)
	\end{equation}
	is an injective homomorphism. 
	As in the previous section, we obtain by
	Proposition \ref{prop:trhull-next}(\ref{prop:trhull-next2}) that $T$, and so also $\Pi(T)$, are isomorphic to $\olS/\oltheta$.
	Furthermore, since $\Pi(S)$ is an ideal in $\olS$, we clearly have 
	$\olf_\chi(\omega^T_t)\in \Pi(S)$ for every $\chi\in \Pi(S)$.
	Combining these facts, we can deduce from (\ref{equ:olfchi}) and (\ref{equ:olphi}) that, for any $s\in S$, the function
	\begin{equation*}
		f_s\colon  T\to \Ker \theta,\quad 
		f_s(t) = \xi(t \ran(\theta(s))) \,s\, \xi^{-1}(t \theta(s))
	\end{equation*}
	is well defined, and the function
	\begin{equation}
		\label{eqn:fi}
		\phi\colon S \to \lwr{\Ker\theta}{T},\quad 
		\phi(s) = (f_s,\theta(s))	
	\end{equation}
	is an injective homomorphism.
	
	Now, for every $s\in S$, let us define $h_s$ to be the restriction of $f_s$ to the principal left ideal $T\ran(\theta(s))$ of $T$.
	First we notice that $h_s\in P^\eta_{K,T}$ for any $s\in S$.
	For, if $t\in T\ran(\theta(s))$ then, by definition, 
	\begin{equation*}
		h_s(t) = f_s(t) = \xi(t)\,s\,\xi^{-1}(t\theta(s))
	\end{equation*}
	and
	\begin{equation*}
		\theta(h_s(t)) = t\theta(s)(t\theta(s))^{-1} =
		\ran(t\ran(\theta(s)) = \ran(t).
	\end{equation*}
	Thus we can define the function
	\begin{equation*}
		\psi\colon S \to \Hwrp{\Ker\theta}{T},\quad 
		\psi(s) = (h_s,\theta(s)).
	\end{equation*}	
	First we show that $\psi$ is a homomorphism. 
	Consider arbitrary elements 
	$\psi(q) = (h_q,\theta(q))$ and $\psi(s) = (h_s,\theta(s))$ in $\Hwrp{\Ker\theta}{T}$.
	Due to the facts that $\phi$ in (\ref{eqn:fi}) is a homomorphism, $h_s$ is a restriction of $f_s$ for every $s\in S$, and the binary operation in the first component of each of $\lwr{\Ker\theta}{T}$ and $\Hwrp{\Ker\theta}{T}$ is `pointwise' multiplication of functions into $K$, it suffices to check that the domains of the first components 
	$h_{qs}$ of $\psi(qs)$ and $h_q \oplus (\theta(q)\cdot h_s)$ of $\psi(q)\psi(s)$ 
	coincide.
	By definition, $\Hdom{h_{qs}} = T\ran(\theta(qs))$, and
	\begin{eqnarray*}
		\Hdom{(h_q \oplus (\theta(q)\cdot h_s))} 
		&=& T\ran(\theta(q)) \cap T\ran(\theta(s))(\theta(q))^{-1} = 
		T\theta(\ran(q)) \cap T\theta(\ran(s)q^{-1}) \cr 
		&=& T\theta(\dom(\ran(s)q^{-1})) = T\theta(\ran(qs)) = T\ran(\theta(qs))
	\end{eqnarray*}
	since $\ran(q) = \dom(q^{-1}) \ge \dom(\ran(s)q^{-1}) = q\ran(s)q^{-1} = \ran(qs)$.
	Thus $\psi$ is, indeed, a homomorphism.
	
	To verify that $\psi$ is injective, recall that injectivity of $\phi$ is proved in 
	\cite[Theorem 5.3.5]{LawInvSg} by checking that, for any $s\in S$, we have
	\begin{equation*}
		s = \xi^{-1}(\ran(\theta(s))) f_s(\ran(\theta(s))) \xi^{-1}(\theta(s)).
	\end{equation*}
	Since $\ran(\theta(s))\in T\ran(\theta(s)) = \Hdom h_s$ this implies that
	\begin{equation*}
		s = \xi^{-1}(\ran(\theta(s))) h_s(\ran(\theta(s))) \xi^{-1}(\theta(s))
	\end{equation*}
	also holds.
	Thus if $q,s\in S$ such that $\psi(q) = (h_q,\theta(q))$ and $\psi(s) = (h_s,\theta(s))$ are equal elements in $\Hwrp{\Ker\theta}{T}$, that is, 
	$h_q = h_s$ and $\theta(q) = \theta(s)$, then $q = s$ follows, and injectivity of $\psi$ is also verified.
\end{proof}

\begin{Rem}
	{\rm
	It is not difficult to realize that an argument analogous to the second part of the proof of Theorem \ref{ThmM} verifies that the $\lambda$-wreath product $\lwr{K}{T}$ and Houghton's wreath product $\Hwr{K}{T}$ are isomorphic to each other for any inverse semigroups $K$ and $T$.
	More precisely, the function}
	\begin{equation*}
		\Psi\colon \lwr{K}{T} \to \Hwr{K}{T},\quad 
		\Psi(f,t) = (f|_{Tt^{-1}},t),
	\end{equation*}
	{\rm
	where $f|_{Tt^{-1}}$ stands for the restriction of $f$ to the domain $Tt^{-1}$, can be seen to be an injective homomorphism.
	Finally, surjectivity is shown as follows.
	For every element $(h,t)$ in $\Hwr{K}{T}$, the pair $(\olh,t)$ where 
	$\olh\colon T\to K$ is defined by the rule 
	$\olh(x) = h(x\ran(t))\ (x\in T)$ belongs to $\lwr{K}{T}$ and 
	$\Psi(\olh,t) = (h,t)$.
	
	From the point of view of this work, the advantage of Houghton's wreath product is in the feature that its modified version corresponding to a normal extension triple
	$(K,\eta,T)$ is more natural than that in the respective $\lambda$-wreath product.

	The fact that the $\lambda$-wreath product $\lwr{K}{T}$ and Houghton's wreath product $\Hwr{K}{T}$ are isomorphic was observed also by J.\,Kadourek (private communication) when reading the prerpint of this work.
	This contributed to the author's decision to include this remark in the final version of the paper.}
\end{Rem}

\end{document}